\documentclass[]{interact}

\usepackage{epstopdf}
\usepackage{subfigure}

\usepackage[numbers,sort&compress,merge]{natbib}
\bibpunct[, ]{(}{)}{,}{n}{,}{,}

\theoremstyle{plain}
\newtheorem{theorem}{Theorem}[section]
\newtheorem{lemma}[theorem]{Lemma}

\newtheorem{proposition}[theorem]{Proposition}

\theoremstyle{definition}

\theoremstyle{remark}

\usepackage{lineno,hyperref}
\usepackage{footmisc}
\usepackage{booktabs}
\usepackage{xcolor}
\usepackage{multirow}
\usepackage{caption}
\modulolinenumbers[5]

\usepackage{amsthm,amsmath}
\usepackage{amsfonts}
\usepackage{mathtools}
\begin{document}

\articletype{ARTICLE TEMPLATE}
\title{On the Voigt profile and its dual}

\author{
\name{Massimo~Cannas\textsuperscript{a}\thanks{CONTACT prof. M.~Cannas. Email: massimo.cannas@unica.it} }
\affil{\textsuperscript{a} University of Cagliari, Italy}
}

\maketitle

\begin{abstract}
The Voigt profile is the density obtained from the convolution of a Gaussian and a Cauchy and it is widely used in atomic and molecular spectroscopy. We show that the Voigt profile is a scale mixture of Gaussian distributions, with mixing Levy distribution. A consequence of this result is that there exists a dual of the Voigt distribution, which is itself a normal scale mixture. Both the Dual Voigt and its mixing are transformations, via truncation and reflection, of the Normal and Levy random variables. We discuss the dual Voigt characteristics, propose algorithms for parameter estimation and outline further developments.
\end{abstract}


\section{Introduction}
Let $X~\sim Cauchy (\mu_X, \gamma)$ and $Y~\sim Normal(\mu_Y,\sigma^2)$ be independent random variables. Then the law of 
\begin{equation}
\label{voigtconv}
U=X+Y
\end{equation}
 is known as the Voigt distribution of parameters ($\mu,\gamma,\sigma^2$) where $\mu=\mu_X+\mu_Y$ and $\sigma, \, \gamma>0$. When $\mu=0$ we talk of the centered Voigt distribution. The Voigt distribution is ubiquitous in several areas of physics because it well describes several broadening mechanisms up to suitable proportionality constants. In spectroscopy it is customarily employed the term Voigt \emph{profile} to refer at the empirical density arising from the superposition of Doppler and Gaussian broadening.

\section{Voigt as a normal scale mixture}
In this paper the expression ``normal scale mixture" means a variance mixture of a Normal distribution. If $U$ is a normal scale mixture its density can be written as:
\begin{equation}
\label{nsm}
 p(u)= \int_{0}^{\infty} \frac{1}{\sqrt{2\pi v}}\exp\left ( \frac{-u^2}{2v} \right )f(v)d(v)
\end{equation}
This can be concisely written as $N(\mu, V)$ or even more compactly with the product formula $ZV^{1/2}$, where $V$ is a probability distribution concentrated on $(0, +\infty)$, i.e. $V$ has support $\lbrack c, +\infty)$ with $c\ge0$.  

Herein we show that the Voigt is a normal scale mixture. We give a proof via stochastic representation. The proof was inspired by the remarkably simple proof of the scale-mixture representation of the Laplace distribution given by Ding and Blitzstein \cite{ding}. In the following we write $A\sim B$ if random variables $A$ and $B$ have the same distribution. We recall the following lemma:
\begin{lemma}
\label{lemma}
Let $Z$, $Z_1$ and $Z_2$ be standard normal variables. Let also $L$ and $C$ be Levy and Cauchy distribution, respectively. The following representations hold:
\begin{itemize}
\item[a)] $\sigma^2+\frac{\gamma^2}{Z^2}\sim L(\sigma^2,\gamma^2)$ 
\item[b)] $\frac{\sigma_1 Z_1}{\sigma_2 Z_2} \sim C(0,\frac{\sigma_1}{\sigma_2})$
\end{itemize}
\end{lemma}
We stated the lemma without proof but we observe that the first result is used for simulation of stable distributions, see e.g., Chambers et al. \cite{chambers}. The second recalls one standard way of generating the Cauchy distribution, as a ratio of normal distributions. We can now prove that the Voigt is a normal scale mixture as follows. Using stochastic representation, the previous proposition is the product formula $U=ZL^{1/2}$ and the generalization for the uncentered Voigt is simply $U=NL^{1/2}$. 

\begin{proposition}[voigt as normal scale mixture]
\label{nsmvoigt2}
Let $U\sim Voigt(\mu,\gamma,\sigma^2)$. Then $U\sim Normal(\mu,L)$ with mixing distribution $L\sim Levy(\sigma^2,\gamma^2)$. 
\end{proposition}
\begin{proof}
We first consider the centered Voigt profile ($\mu=0$). We show that $N(0, L) \sim \sqrt{L} Z$ is decomposable as the sum of a Cauchy and Normal distribution.
Proposition \ref{lemma} a) implies that

$$\sqrt{L}\, Z = \sqrt{\sigma^2 + \frac{\gamma^2}{Z_1^2}}\, Z$$

Conditioning on $Z_1$ yields

$$ \sqrt{L} Z \, | \, Z_1 \sim N\left(0, \frac{\gamma^2}{z_1^2} + \sigma^2\right)$$

Proposition \ref{lemma} b) implies that

$$  \gamma\frac{Z_2}{Z_1}+\sigma Z_3  \, | \,Z_1 \sim N\left(0, \frac{\gamma^2}{z_1^2} + \sigma^2\right)$$

So, unconditionally

$$ \gamma\frac{Z_2}{Z_1}+\sigma Z_3 \sim \sqrt{L} Z $$

The generalization to any $\mu$ is easily obtained by considering $Z_1+\mu$ and $Z_2 + \mu$ in place of $Z_1$ and $Z_2$. 
\end{proof}

We discuss herein some interpretations of the product formula: 
\begin{equation}\label{productformula}
U=ZL^{1/2}
\end{equation}

\subsection{Composition of stable laws}
When $\sigma=0$ the product formula becomes $C=Z(IG)^{1/2}$ and connects the three \emph{strictly} stable laws with closed form density. In fact, the normal with zero mean is strictly stable with stability parameter $\alpha=2$, the unshifted Levy is an inverse gamma, which is strictly stable with $\beta=1/2$ and their product is a Cauchy, which is strictly stable with parameter $\alpha \beta=1$. More generally, if $X$ and $Y$ are strictly stable distributions with stability parameters $\alpha$ and $\beta$ then $X^{1/\alpha}Y$ is strictly stable with parameter $\alpha\beta$ (see Feller VI.2, example h, pag 176, also pag 348 and pag 596  \cite{feller}). 

When $\sigma>0$ the previous property is no longer valid since the Levy is not \emph{strictly} stable. By proposition \ref{nsmvoigt2} the product is a Voigt, which is not stable since it is the sum of two stable distributions with different exponents. In general, the sum of stable laws with different exponents cannot be stable (see Nolan \cite{nolan}, page 264). There is no simple formula for the sum of stable variables with different exponents. Otiniano et al. \citeyear{otiniano} showed that the density can be expressed in terms of the Fox $H$-function.

\subsection{Stochastic processes}
The product formula $U=ZL^{1/2}$ can also be interpreted in the language of stochastic processes by saying that the Voigt process arises as a standard brownian motion subordinated to a Levy distribution. In fact, let $X(t)$ be a standard brownian motion with transition density $p_t(x)=(2\pi t)^{-1/2} e^{-1/2 x^2/t}$ and let $T(t)~\sim Levy(\sigma^2, t^2/2)$. When $\sigma=0$, the subordinated process $X(T(t))$ has a Cauchy transition density (see Feller X.7 pag 348) and it is therefore known as Cauchy process. When $\sigma>0$, the transition density is given by the product formula in eq \ref{productformula} so the process $U(T(t))$ has a Voigt transition density.

\subsection{Voigt as prior predictive distribution}
The normal scale mixture representation $U=ZL^{1/2}$  arises in Bayesian inference when we consider a normal model $N(0,v)$ with unknown variance $v$ and we assume that $V\sim Levy(\sigma^2,\gamma^2)$ is the prior distribution for the variance. The latter can be thought of as an informative prior excluding variance values less than a known threshold $\sigma^2$. Kim et al.~\cite{kim} used this distribution for modeling long-range internet traffic. Indeed, using Bayes' rule the posterior distribution is
\begin{equation*}
\label{bayes}
f(v | u)=\frac{f(u | v) f(v)}{\int f(u | v) f(v)dv}
\end{equation*}
and the mixture representation of the Voigt appears in the denominator, i.e. the Voigt is the \textit{prior predictive distribution} for the Normal-Levy model. 

\section{The Dual Voigt distribution}
Dual (or adjunct) densities were introduced by Rossberg \cite{rossberg} in a study of positive definite functions. The classic example is the self duality of the normal distribution (example 2.4 in Rossberg\cite{rossberg}; I.J. Good \cite{good}), which is deemed as a curious phenomenon by Feller (\cite{feller} Vol. II, pag 504). Another pair of well-known dual variables are the Cauchy and Laplace distributions. A list of dual densities of common distributions was provided by Nadaraja \cite{nadaraja}. It turns out that most dual densities are not common distributions but they often have more tractable characteristics than the original variables. Not all random variables admit a dual but, when the dual exists, its probability density $p'$ is proportional to the characteristic function of the original variable. More precisely, the following relations hold \cite{rossberg} for a random variable with density $p$ and characteristic function $\psi$:
\begin{equation}
\label{fundrel}
\psi(t)=\frac{p'(t)}{p'(0)} \qquad \psi'(t)=\frac{p(t)}{p(0)} \qquad 2\pi p(0)p'(0)=1
\end{equation}

 The next proposition lists the main characteristics of the Dual Voigt.

  \begin{proposition}
  \label{pippo}
  Let $U$ be a centered Voigt$(\gamma, \sigma)$. Then 
  \begin{itemize}
  \item[i)] The dual of $U'$ exists and belongs to the 2-parameter exponential family
  \item[ii)] Let $X\sim Norm(-\gamma/\sigma^2, 1/\sigma^2)$, then
  \begin{equation}\label{3.1.2}
|U'| = X[X>0] \, \iff \,  U'=
\begin{cases} 
        (X | X>0)  & \textrm{with probability }1/2 \\
      - (X | X>0) &  \textrm{with probability } 1/2  \\
   \end{cases}
     \end{equation}
\item[iii)]
\begin{equation}\label{3.1.3}
E[(U')^{n}]=
\begin{cases} 
        E\left( X | X>0 \right)^{n}&  n\, \textrm{even}\\
       0         &   n \,  \textrm{odd}  \\
   \end{cases}
  \end{equation}
  \end{itemize}
  \end{proposition}
  \begin{proof}
    \begin{itemize}
  \item[i)] 
    Following Rossberg, the dual $U'$ exists when the density function $p(u)$ is non negative, bounded and integrable. The Voigt density is clearly non negative and integrable so it remains to show it is bounded. Since $p(u)<p(0)$ it is enough to show that $p(0)$ is finite, which is proved in the Appendix. 
   We note that the first relation (\ref{fundrel}) implies that the density function is $p'(u)=p'(0)\psi(u)$. The area behind $\psi(u)$ is
\begin{align*}
A= \int_{-\infty}^{+\infty} \psi(u) du = \int_{-\infty}^{+\infty} e^{-\gamma |u| - \sigma^2 u^2/2} du&\\
 \textrm{Completing the square in the exponential we have}\\[4pt]
         = \frac{\sqrt{2\pi}}{\sigma}e^{\gamma^2/2\sigma^2} \int_{-\infty}^{+\infty} \frac{\sigma}{\sqrt{2\pi}}&e^{-\gamma^2/2\sigma^2 -\gamma |u| - \sigma^2u^2/2} du \\
         = \frac{\sqrt{2\pi}}{\sigma}e^{\gamma^2/2\sigma^2} \int_{-\infty}^{+\infty} \frac{\sigma}{\sqrt{2\pi}}&e^{-\sigma^2/2(u^2 + 2\gamma/\sigma^2 |u| + \gamma^2/\sigma^4)} du \\                          
         = \frac{\sqrt{2\pi}}{\sigma}e^{\gamma^2/2\sigma^2} \int_{-\infty}^{+\infty} \frac{\sigma}{\sqrt{2\pi}}&e^{-\sigma^2/2(|u| + \gamma/\sigma^2)^2} du \\
                                    = \frac{\sqrt{2\pi}}{\sigma}e^{\gamma^2/2\sigma^2}2(1-\Phi(\gamma/\sigma))& 
 \end{align*}
 Thus the density function is 
   \begin{align}
   p'(u)&= \frac{1}{A}\psi(u) \nonumber \\ 
         &= \frac{1}{2\left(1-\Phi(\gamma/\sigma)\right)}\frac{\sigma}{\sqrt{2\pi}} e^{-\frac{\gamma^2}{2\sigma^2}}e^{-\gamma |u| -\frac{\sigma^2}{2} u^2} \label{dvdensity}
   \end{align}

Finally, let $\theta=(\gamma, \sigma^2)$. Then
 $$p'(u )=a(\theta)b(y)\exp\left[ c_1(\theta)d_1(u)+c_2(\theta)d_2(u) \right]$$
where
$$a(\theta)=\frac{e^{-\frac{\gamma^2}{2\sigma^2}}}{2(1-\Phi(\gamma/\sigma)) } \frac{\sigma}{\sqrt{2\pi} }
 \qquad b(y)=1,$$
 \begin{equation*}
c_1(\theta)=-\gamma  \qquad  c_2(\theta) =-\frac{\sigma^2}{2}
 \qquad d_1(u)=|u|  \qquad  d_2(u)=u^2
 \end{equation*}
 so $U'$ belongs to the 2-parameter exponential family.
  \item[ii)]   
We observe that, for $u >0$, the density function is proportional to the density of $X \sim N(-\gamma/\sigma^2, 1/\sigma^2)$ truncated below zero. The density is reflected over the vertical axis so the division by $ 2Pr (X>0)=2(1-\Phi(\gamma/\sigma))$ normalizes the area behind the curve. Thus, the Dual Voigt is a mixture of $(X | X>0)$ and $-(X | X>0)$ with equal weights.
 \item[iii)]
  All odd moments are equal to zero because the dual Voigt density is an even function. As for the even moments we have
\begin{align*}
E[(U')^{2n}]&=\int_{-\infty}^{\infty} u^{2n} \frac{1}{2\left(1-\Phi (\gamma/\sigma)\right)}\frac{\sigma}{\sqrt{2\pi}}e^{-\frac{\sigma^2}{2}\left(|u|+\frac{\gamma}{\sigma^2}\right)^2}du\\
&=\int_{0}^{\infty} u^{2n} \frac{1}{\left(1-\Phi (\gamma/\sigma)\right)}\frac{\sigma}{\sqrt{2\pi}}e^{-\frac{\sigma^2}{2}\left(u+\frac{\gamma}{\sigma^2}\right)^2}du
\end{align*}
 The last expression is the moment of order $2n$ of a $N(-\gamma/\sigma^2, 1/\sigma^2)$ truncated from below at zero.  
  \end{itemize}
  \end{proof}

When $\gamma \to 0$ the dual Voigt density is increasingly bell-shaped while for $\gamma \to +\infty$ is increasingly spiked, see Figure~\ref{plotdual}. The ratio $\gamma/\sigma$ plays a pivotal role for the shape of the Dual Voigt. The $Normal (-\gamma/\sigma^2, 1/\sigma^2)$ has the right inflection point equal to $(-\gamma + \sigma)/\sigma^2 <0$ for $\gamma < \sigma$. We can conclude that the Dual Voigt is log-convex for $\gamma < \sigma$ and neither log-convex nor log-concave otherwise. In the limiting case $\gamma=0$ both the Voigt and its dual are the same normal random variable $N(0,1/\sigma^2)$ and we recover the self-duality of the normal distribution. The dual voigt peak in zero depends both on the ratio $\gamma/\sigma$ and on $\sigma$ (or $\gamma$). We have 
$p'(0) = \tfrac{\sigma}{2R(t)}$ where $R(t) =\frac{1-\Phi(t)}{\phi(t)}$ is the Mills ratio. In the case $t = \gamma / \sigma = 1$ the peak is lower than the standard normal peak if $\tfrac{\sigma}{2R(1)} < \phi(0) = \tfrac{1}{\sqrt{2\pi}}$ which is true when $\sigma =\gamma< 0.523$. Conversely, for $\sigma \geq 0.523$, the peak of the Dual Voigt exceeds the Normal one, see~Figure~\ref{plotdual} (left).


 We close this section with a practical formula for the Dual Voigt second and fourth moment, obtained using a recursive formula for the moments of a truncated normal distribution given in Horrace \cite{horrace}:
 \begin{align*}
   E(U'^{2})&=\frac{\gamma^2}{\sigma^4}+\frac{1}{\sigma^2} - \frac{\gamma}{\sigma^3} \left(\frac{ \phi(\gamma/\sigma)}{1-\Phi(\gamma/\sigma)} \right)\\[6pt]
  E(U'^{4})&=\frac{\gamma^4}{\sigma^8}+\frac{6\gamma^2}{\sigma^6} +\frac{3}{\sigma^4} 
  - \left(\frac{5\gamma}{\sigma^5} + \frac{\gamma^3}{\sigma^7}\right) \left( \frac{\phi(\gamma/\sigma)}{1-\Phi(\gamma/\sigma)} \right).
 \end{align*}
 
 Since $E(U')=0$, the first expression gives the dual Voigt variance.

\section{The Dual Voigt as a normal scale mixture}
Gneiting \cite{gneiting} discussed specifically dual densities of normal scale mixtures. He showed that, whenever the dual of a normal scale mixture exists, it is itself a normal scale mixture. He also provided a formula for the dual mixing and several examples of dual pairs. Using his results we can prove the following
\begin{proposition}
\label{dvnsm}
Let $U\sim Voigt(0,\gamma,\sigma^2)$. Then the Dual Voigt $U'$ is a normal scale mixture, $U'\sim Normal(\mu,V')$, with mixing distribution:
\begin{align*}
V'&=2/\sigma^2 - L[ L < 2/\sigma^2] \\
L&\sim Levy\left(1/\sigma^2,\gamma^2/\sigma^4\right)
\end{align*}
\end{proposition}
\begin{proof}
$U$ is a zero centered normal scale mixture by proposition \ref{nsmvoigt2}. Since $p(0)<\infty$, it follows from Gneiting~\cite{gneiting} that its dual $U'$ is also a normal scale mixture. Thus, we have $U'=(V')^{1/2}Z$ for some mixing distribution $V'$. The mixing distribution can be found applying equation~7 in \cite{gneiting}: 

\begin{align}
p'(v)&= \frac{1}{p(0)}\frac{1}{(2\pi)^{1/2}  v^{3/2}} p\left(\frac{1}{v}\right)\nonumber\\
&= \frac{1}{p(0)}\frac{1}{(2\pi)^{1/2}  v^{3/2}} \frac{\sqrt{\gamma^2/2}}{\sqrt{\pi}} \frac{1}{\left(\frac{1}{v}-\sigma^2\right)^{3/2}}
\exp \left(-\frac{\gamma^2}{2} \frac{1}{(\frac{1}{v}-\sigma^2)} \right)I_{(\sigma^2,+\infty)}\left(\frac{1}{v}\right)\nonumber\\
&= \frac{\gamma}{2\pi p(0)}\frac{1}{\left(1-v\sigma^2\right)^{3/2}}\exp \left(-\frac{\gamma^2}{2} \frac{v}{(1-v\sigma^2)} \right)I_{(0,  \frac{1}{ \sigma^2})}(v)\nonumber\\
\quad \Aboxed{  p(0)&=\frac{1}{2\pi p'(0)}}\nonumber \\
&=\frac{1}{2\left(1-\Phi(\gamma/\sigma)\right)}\frac{\gamma/\sigma^2}{\sqrt{2\pi}}\left(\frac{1}{\sigma^2}-v\right)^{-3/2}e^{-\frac{(\gamma/\sigma^2)^2}{2\left(\frac{1}{\sigma^2}-v\right)}}I_{(0,  \frac{1}{ \sigma^2})}(v) \label{dvmixing}
\end{align}

In the last passage, we used the value of $p'(0)$ derived in proposition~\ref{pippo}. It remains to show that the mixing distribution is a linear transformation of a truncated Levy distribution.
 Let $L\sim Levy(1/\sigma^2, \gamma^2 /\sigma^4)$. Consider the new variable $Z$ obtained by truncating $L$ above $2/\sigma^2$. Since $ F( 2/\sigma^2)=2\left(1-\Phi(\gamma/\sigma)\right)$ we have $f(z)=\frac{1}{F(2/\sigma^2)}f(z)I_{(1/\sigma^2, 2/\sigma^2)}(z)$. Finally, the transformation $\hat{V}=2/\sigma^2 - Z$ mirrors the density of $Z$ over the truncation point and gives the density \ref{dvmixing}. Thus, we have $\hat{V} = 2/\sigma^2 - L \lbrack L < 2/\sigma^2\rbrack $ 
\end{proof}
The construction of the Dual Voigt mixing is depicted in Figure~\ref{construction} for $\sigma=1$.


Finally, we compare the dual Voigt mixing density (i.e., the Levy) with the Voigt mixing density in Figure~\ref{plotmixing}. We can notice an inverse relation between the dispersion of the two mixing distributions. Higher values of $\sigma$ are associated with flat Levy density and a small support in the dual Voigt. Similarly, for fixed $\sigma$, higher values of $\gamma$ lead to large dispersion in the Levy and low dispersion in the dual Voigt. The inverse relation between the dispersion of the Voigt mixing and the dual Voigt mixing is no coincidence. A notable example is the $N(0,a)$ whose dual is the $N(0,1/a^2)$. In general, dual densities exhibit a sort of uncertainty principle in that the product of their variances cannot be made arbitrarily small. Dreier~\cite{dreier} showed that the product of the variances is bounded below by $0.527<\Lambda<0.867$ and also proved that $\Lambda$ can be made arbitrarily large via an artificial example (see \cite{dreier}, Lemma 3.1). Gneiting \cite{gneiting2} showed that $\Lambda \ge 1$ when both densities are normal scale mixtures, with equality holding for standard normal. The pair (Voigt, dual Voigt) is quite pathological and $\Lambda$ is undefined due to the Voigt variance being undefined. 
However, the graph suggests that an inverse variability relation still holds.

\section{Infinite divisibility of the Voigt and Dual Voigt}
A random variable is infinitely divisible (id) if it can be expressed as the sum of an arbitrary number of independent and identically distributed variables. By Feller, (see pag. 425 \cite{feller}) if the variance distribution of a normal scale mixture is id then the mixture is id. Barndorff et al \cite{barndorff} used this fact to prove the infinite divisibility of the hyperbolic distribution. An immediate corollary of the formula $U=ZL^{1/2}$ is then that the Voigt distribution is id since the mixing distribution is a Levy, which is an id distribution. Of course, the infinite divisibility of the Voigt immediately follows from the definition. 

It is natural to ask whether the dual Voigt is infinitely divisible. We observe that the support of $V'$ is bounded and thus $V'$ is not id (see Feller~\cite{feller}, page 177). However, the infinite divisibility of the mixing is only a sufficient condition for the infinite divisibility of the associated normal scale mixture and Kelker~\cite{kelker} provided an example of an id mixture where the mixing is not id. Kelker (see~\cite{kelker}, theorem 2) also showed that if the mixing $G$ is bounded and $G(b)=1$ for some $b$ then the mixture is not infinitely divisible. This implies that the dual Voigt is not infinitely divisible for parameter values leading to modal densities greater than one. Moreover, Wolfe~\cite{wolfe} showed that a necessary condition for $G$ to be the mixing of an id mixture is that the tails of $G$ satisfy an exponential lower bound for large $b$. Wolfe's result implies that a bounded distribution cannot be the mixing distribution of an id normal variance mixture and thus we conclude that the Dual Voigt is not id.

\section{Parameter estimation for the Dual Voigt} 
In this section we briefly consider parameter estimation for the Dual Voigt. We first show the joint sufficiency of the sample absolute sum and sum of squares:
\begin{proposition}
\label{suffstat}
Let $(Y_1, \cdots, Y_n)$ be a random sample from the dual of a Voigt$(\gamma, \sigma^2)$. Then
$$ S_1(Y_1, \cdots, Y_n)= \sum_i | Y_i |,  \qquad S_2(Y_1, \cdots, Y_n)= \sum_i Y_i^2$$
is a pair of asymptotically normal, minimal sufficient and complete statistics.
\end{proposition}
\begin{proof}
 It follows from proposition~\ref{pippo} i) that
\begin{equation}
 S_1=\sum_i d_1(Y_i)=\sum_i  |Y_i | \qquad  \textrm{and} \qquad S_2=\sum_i d_2(Y_i)=\sum_i Y_i^2
 \end{equation}
is a pair of complete minimal sufficient statistics. Asymptotic normality of $(S_1,S_2)$ follows immediately from Central Limit Theorem as both $S_1$ and $S_2$ are linear functions of the sample moments of $|Y|$ and both moments have finite variance (see \ref{3.1.2}).
\end{proof}
 We note that the sample mean and variance of $|Y|$ are also jointly sufficient since they are linear functions of $S_1$ and $S_2$,  respectively. 

We now present a simple procedure for parameter estimation. Let $(\gamma, \sigma^2)$ denote the unknown parameter vector and let $(y_1, \cdots, y_n)$ be an iid sample from the dual of $U\sim Voigt(\gamma, \sigma^2)$. Closed-form solutions of the moment or maximum likelihood equations are clearly not available due to the error function in the density of the Dual Voigt. Previous results connecting $U'$ to the truncated normal distribution suggest that we can use a sample transform strategy. Let $(y_1, \cdots, y_n)$ be an iid sample from $U'$, then proposition~\ref{pippo} i) implies that $(t_1, \cdots, t_n)$ where $t_i= |y_i | $

is a sample from 
$$(X\,|\, X>0)\sim T \sim TruncNorm(0, +\infty, -\gamma/\sigma^2, 1/\sigma^2).$$

where the last two parameters are the mean and variance of the underlying normal distribution. Thus, parameter estimation can be performed by first estimating the last two parameters of $T$ with $e_1$ and $e_2$, and then mapping back the estimates to $\gamma$ and $\sigma$ via 
 \begin{equation}
 \label{transf}
  \widehat{\gamma}=-e_1\cdot (e_2)^{-1} \qquad \textrm{and} \qquad \widehat{\sigma}^2=(e_2)^{-1}
 \end{equation}

In the first stage, it is necessary to find the parameters of a truncated normal distribution (with known truncation values). Maximum Likelihood (ML) and Method of Moments (MoM) are typically found using iterative optimization algorithms \cite{cohen}, e.g., the Nelder-Mead algorithm implemented by function \texttt{optim} in the base \textsf{R} package. We present estimates obtained in this way in Table~\ref{tabest} when the true parameter values are $\gamma=\sigma=1$. We highlight that the absolute bias and standard deviation of the estimates are very similar for parameter values randomly chosen in $[0.2, 4]$. Overall, the procedure shows a satisfactory performance. Occasionally, we observed failure to reach the minimum and inconsistent values, particularly with small sample sizes $(n=50, 100)$.

 \begin{table}[h!] 
\centering 
\begin{tabular}{rrr cc cc cc } 
\toprule 
     \multicolumn{3}{c}{$n$} & \multicolumn{2}{c}{$\hat{\gamma}$} & \multicolumn{2}{c}{$\hat{\sigma}$} & \multicolumn{2}{c}{$st. dev.$}\\ 
  \cmidrule(rl){4-5}   \cmidrule(rl){6-7}  \cmidrule(rl){8-9} 
   & &  & \textsc{ML}& \textsc{MoM}  & \textsc{ML}& \textsc{MoM} & \textsc{ML}& \textsc{MoM}  \\
\midrule 

       & 100 &  & 0.68& 0.71&0.70&0.70& 0.62 & 0.63\\ 

    \cmidrule(rl){1-9} 
    & 500 &   &0.94 &0.94&1.03&1.03& 0.28& 0.28 \\

\cmidrule(rl){1-9} 
    & 1000 &   &0.98  &0.98&1.01&1.01& 0.19& 0.19 \\

\cmidrule(rl){1-9} 
    & 5000 &   &1.00  &1.00&0.99&0.99& 0.08 & 0.08 \\
\bottomrule 
\end{tabular}
\caption{Average point estimates and standard deviation for parameters of the Dual Voigt$(1,1)$ calculated over $100$ simulated datasets. Method of moments (MoM) and maximum likelihood (ML) estimates obtained with the Nelder-Mead algorithm implemented by function \texttt{optim} in base \textsf{R} package.
}
\label{tabest} 
\end{table}

\section{Discussion}
Normal scale mixtures are a classic topic in distribution theory. Several papers focused on showing that common random  variables admit a normal scale mixture representation. Andrews and Mallows \cite{mallows} showed that $t-$student and Laplace distributions are normal scale mixtures with an inverse Gamma and Gamma random variable mixing distribution, respectively. The former result was proved later by Ding and Blitzstein \cite{ding} using elementary methods. More recently, Kotz and Ostrovskii~\cite{kotz} showed that the Linnik distribution is a normal scale mixture. Gneiting \cite{gneiting} focused on dual of normal scale mixtures showing that normal scale mixtures admit dual of the same type. In this paper we add to the topic showing that the Voigt profile is a normal scale mixture and analyzed its dual counterpart. We showed that the Dual Voigt and its mixing are transformations, via truncation and reflection, of the normal and the Levy random variable respectively and discussed a sample transform procedure for parameter estimation. 

\bibliographystyle{tfp}
\bibliography{mybibfile}

 \appendix

   \section{Existence of the Dual Voigt}
  Let $U$ be a random variable with density $p(u)$. The existence of a dual density is equivalent to $p(u)$ being positive-definite function (see Theorem 2.1 in Rossberg \cite{rossberg}). A positive-definite function is continuous, symmetric and bounded on $\mathbb{R}$. When $U$ is a centered Voigt$(\sigma^2, \gamma)$ the density $p(u)$ is clearly symmetric and continuous so it remains to prove that $p(0)$ is finite:
\begin{align*}
    p(0) &=\int_{\sigma^2}^{\infty}  \frac{1}{\sqrt{2\pi v}}\cdot f(v)dv \qquad (v=\tau+\sigma^2)\\
      &=\int_{0}^{\infty}  \frac{1}{\sqrt{2\pi (\tau+\sigma^2)}}\cdot \frac{\sqrt{\gamma^2/2}}{\sqrt{\pi}} \frac{1}{\tau^{1/2+1}} 
      \exp\left( -\frac{\gamma^2}{2}\frac{1}{\tau}\right)d\tau \nonumber \\[4pt]     
            &\le\int_{0}^{\infty}  \frac{1}{\sqrt{2\pi \tau}}\cdot \frac{\sqrt{\gamma^2/2}}{\sqrt{\pi}} \frac{1}{\tau^{1/2+1}} 
      \exp\left( -\frac{\gamma^2}{2}\frac{1}{\tau}\right)d\tau \nonumber \\[4pt]     
      &=\frac{1}{\sqrt{2\pi }}\frac{1}{\sqrt{\pi }} \frac{1}{\sqrt{\gamma^2/2}}\cdot \int_{0}^{\infty} \frac{\gamma^2/2}{\Gamma(1)} \frac{1}{\tau^{1+1}} \exp\left( -\frac{\gamma^2}{2}\frac{1}{\tau}\right)d\tau \nonumber \\[4pt] 
      &=\frac{1}{\gamma \pi}\nonumber
\end{align*}

after rearranging to isolate the $\Gamma(1/2,\gamma^2/2)$ density.

   \section{Code for Dual Voigt random samples and parameter estimation}
    The code for sampling and estimating parameters of the Dual Voigt is available via this public url / online supplementary material. The following pseudocode describes an accept-reject algorithm for random sampling from a Dual Voigt:
   \begin{enumerate}
   \item Sample candidate $c$ from $N(-\gamma/\sigma^2, 1/\sigma^2)$
   \item If $c>0$ set $x=c$
   \item If $c \le -2\gamma/\sigma^2$ set $x=c+2\gamma/\sigma^2$
   \item else reject
   \end{enumerate}
The algorithm exploits the proportionality of the dual Voigt density to the tail of the normal distribution. A simple strategy is then to accept normal random variates when they fall in the tail and reject them otherwise. This sampling strategy has probability of acceptance equal to $(1-\Phi(\gamma/\sigma))$ so it is not efficient when the ratio $\gamma/\sigma$ is large. 
 
  \begin{figure}[p]
  \centering
  \includegraphics[width=1.1\linewidth]{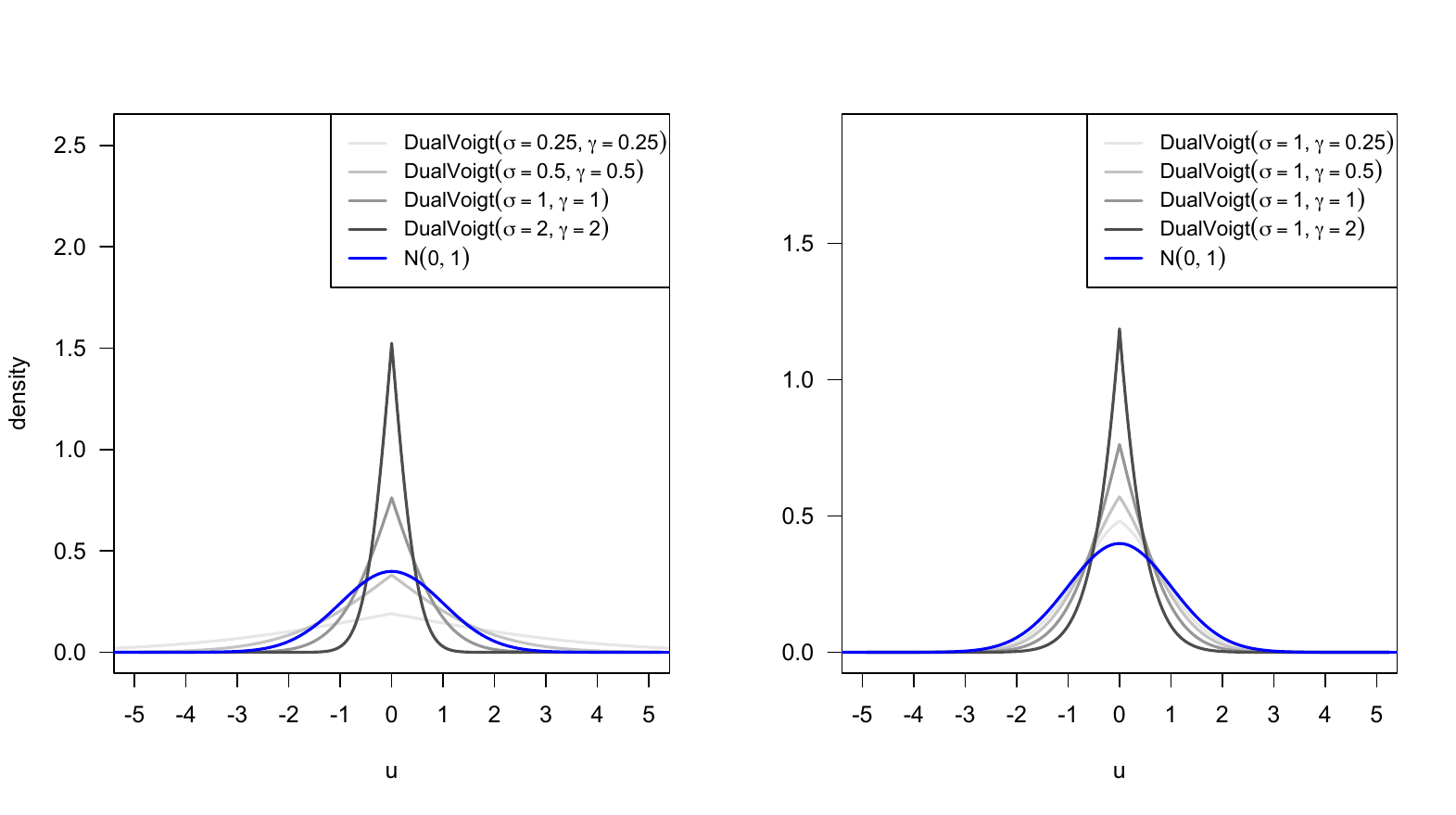}
\caption{The dual Voigt density for $\sigma=\gamma$ pairs (left) and for $\sigma=1, \gamma=0.25, 0.5, 1, 2$ (right). Increasing values of $\gamma$ are associated with increasingly spiked densities. When $\sigma=\gamma$ the Dual Voigt peak exceeds the normal one when $\gamma<0.523$ (left).}
\label{plotdual}
\end{figure}
\begin{figure}[p]
  \centering
  \includegraphics[width=0.8\linewidth]{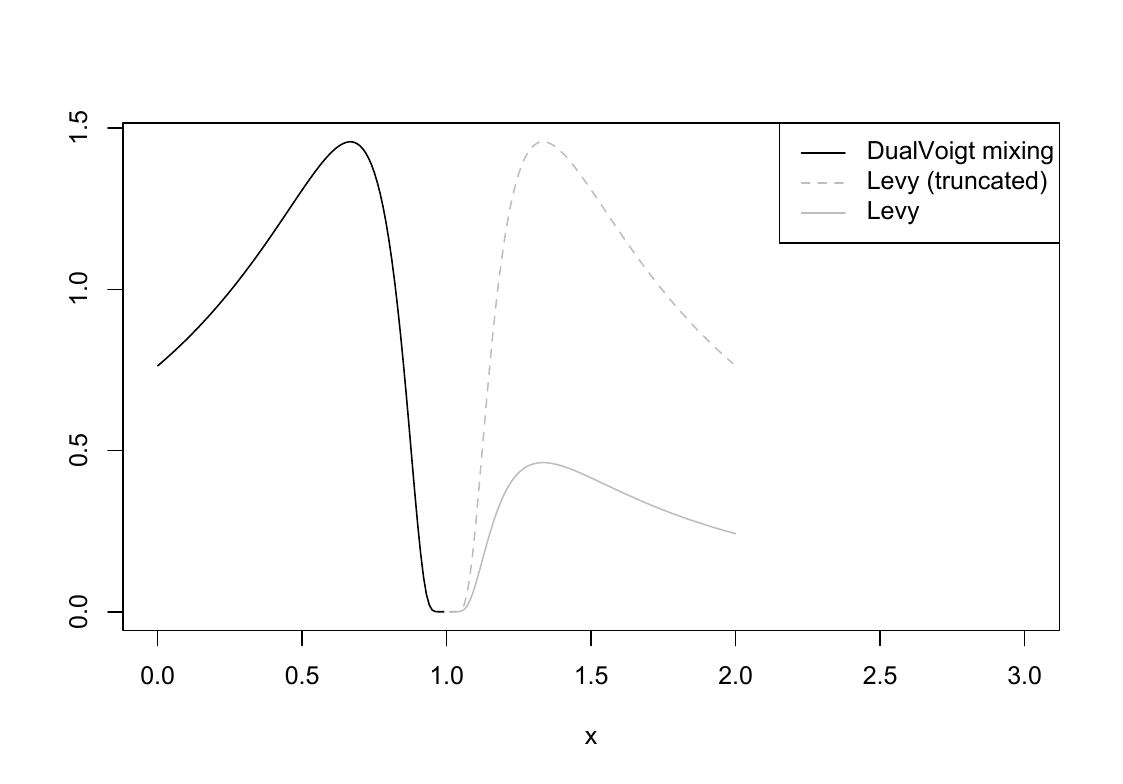}
\caption{Construction of the Dual Voigt mixing ($\sigma=1$). A Levy (gray) is truncated (gray, dashed) and reflected over its minimum (black).}
\label{construction}
\end{figure} 

\begin{figure}[p]
  \centering
  \includegraphics[width=0.7\linewidth]{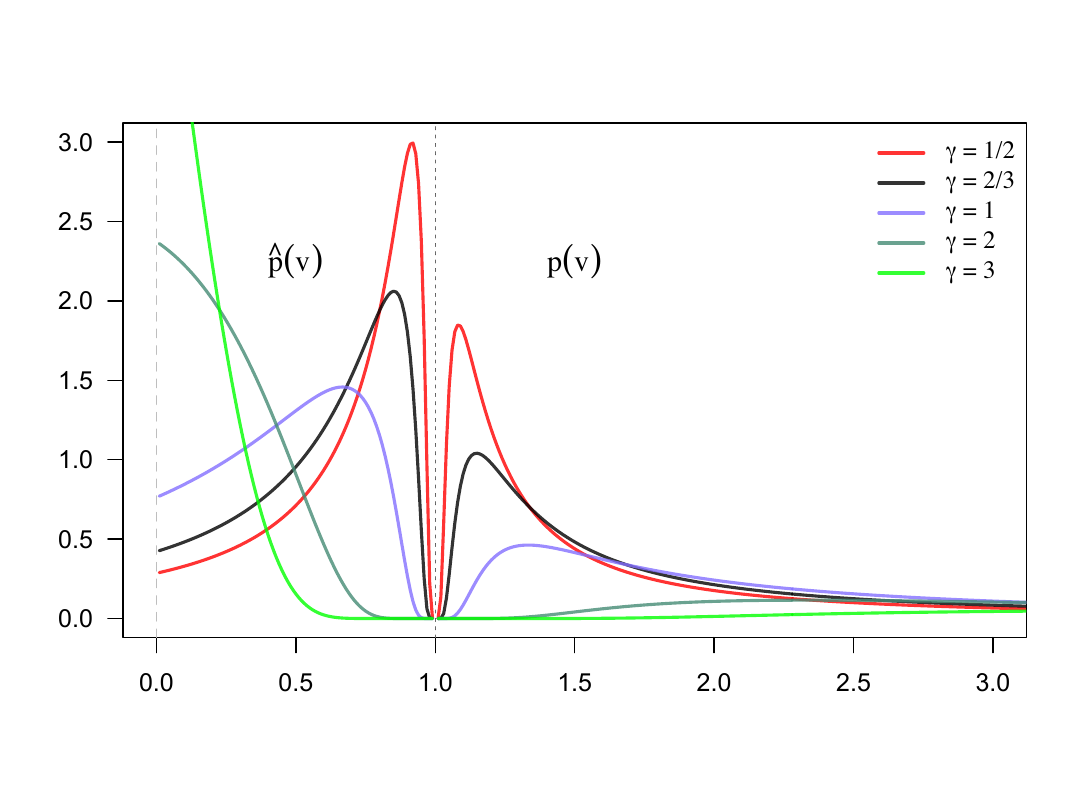}
\caption{Mixing distributions for the dual Voigt (left) and its dual (right). }
\label{plotmixing}
\end{figure}
\end{document}